\documentclass[a4paper,12pt]{article}

\usepackage[french, ngerman, italian, english]{babel}
\usepackage{amsmath,amsthm, amssymb}
\usepackage{setspace}
\usepackage{anysize}
\usepackage{url}
\usepackage{graphicx}
\usepackage[colorlinks=true]{hyperref}
\usepackage{mathptmx}
\usepackage{paralist}
\usepackage{verbatim}
\usepackage[utf8]{inputenc}

\theoremstyle{plain}
\newtheorem{thm}{\bf Theorem}[section]

\newtheorem{prop}[thm]{\bf Proposition}
\newtheorem{lemma}[thm]{\bf Lemma}
\newtheorem{corollary}[thm]{\bf Corollary}

\theoremstyle{definition}

\theoremstyle{remark}
\newtheorem{remark}[thm]{\bf Remark}
\newtheorem{example}[thm]{\bf Example}
\theoremstyle{example}

\def \HF{{\operatorname{HF}_{\bar{A}(\Delta)}}}

\def \chara{{\operatorname{char}}}
\def \cd{{\operatorname{cd}}}
\def \height{{\operatorname{ht}}}

\def \depth{{\operatorname{depth}}}
\def \grade{{\operatorname{grade}}}
\def \Spec{{\operatorname{Spec}}}

\def \Proj{{\operatorname{Proj}}}
\def \mm{{\mathfrak{m}}}
\def \aa{{\alpha}}
\def \bb{{\beta}}

\def \NN{\mathbb N}
\def \PP{\mathbb P}

\def \V{\mathcal V}
\def \J{\mathcal Symb}
\def \M{\mathfrak M}

\def \FD{\mathcal{F}(\Delta)}
\def \D{\Delta}

\def \Id{I_{\Delta}}
\def \Idm{I_{\Delta}^{(m)}}
\def \JD{J(\Delta)}
\def \JDm{J(\Delta)^{(m)}}
\def \AD{\bar{A}(\Delta)}

\begin{document}
\title{Symbolic Powers and Matroids}
\author{ Matteo Varbaro \\
\footnotesize Dipartimento di Matematica\\
\footnotesize Univ. degli Studi di Genova, Italy\\
\footnotesize \url{varbaro@dima.unige.it}}
\date{{\small \today}} 
\maketitle

\abstract 
\noindent We prove that all the symbolic powers of a Stanley-Reisner ideal $\Id$ are Cohen-Macaulay if and only if the simplicial complex $\D$ is a matroid.



\section{Introduction}

Stanley-Reisner rings supply a bridge between combinatorics and commutative algebra, attaching to any simplicial complex $\D$ on $n$ vertices the Stanley-Reisner ideal $\Id$ and the Stanley-Reisner ring $K[\D]=S/\Id$, where $S$ is the polynomial ring on $n$ variables over a field $K$. One of the most interesting part of this theory is finding relationships between combinatorial and topological properties of $\D$ and ring-theoretic ones of $K[\D]$. For instance, it is a wide open problem to characterize graph-theoretically the graphs $G$ for which $K[\D(G)]$ is Cohen-Macaulay, where $\D(G)$ denotes the independence complex of $G$.
In \cite[Theorem 3]{TY}, Terai and Yoshida proved that $S/\Id^m$ is Cohen-Macaulay for any $m\in \NN_{\geq 1}$ if and only if $\Id$ is a complete intersection. Because it is a general fact that all the powers of any homogeneous complete intersection ideal are Cohen-Macaulay, somehow the above result says that there are no Stanley-Reisner ideals with this property but the trivial ones. Since if $S/\Id^m$ is Cohen-Macaulay then $\Id^m$ is equal to the $m$th symbolic power $\Idm$ of $\Id$, it is natural to ask:

{\it For which $\D$ the ring $S/\Idm$ is Cohen-Macaulay for any $m\in \NN_{\geq 1}$?}

The answer is amazing. In this paper we prove that {\it $S/\Idm$ is Cohen-Macaulay for any $m\in \NN_{\geq 1}$ if and only if $\D$ is a matroid} (Theorem \ref{main}). The above result is proved independently and with different methods by Minh and Trung in \cite[Theorem 3.5]{MT}. Matroid is a well-studied concept in combinatorics, and it was originally introduced as an abstraction of the notion of the set of bases of a vector space. The approach to prove the above result is not direct, passing through the study of some blowup algebras related to $\D$. Among the consequences of Theorem \ref{main} we remark Corollary \ref{sci}: {\it After localizing at the maximal irrelevant ideal, $\Id$ is a set-theoretic complete intersection whenever $\D$ is a matroid.}

\section{The result}

In this section we prove the main theorem of the paper. 

\subsection{Definition of the basic objects}

First of all we define the basic objects involved in the statement. For the part concerning commutative algebra and Stanley-Reisner rings, we refer to Bruns and Herzog \cite{BH}, Stanley \cite{St} or Miller and Sturmfels \cite{MS}. For what concerns the theory of matroids, some references are the book of Welsh \cite{We} or that of Oxley \cite{Ox}.
 
Let $K$ be a field, $n$ a positive integer and $S=K[x_1,\ldots ,x_n]$ the polynomial ring on $n$ variables over $K$. Also, $\mm$ is the maximal irrelevant ideal of $S$. We denote the set $\{1,\ldots ,n\}$ by $[n]$. By a {\it simplicial complex} $\D$ on $[n]$ we mean a collection of subsets of $[n]$ such that for any $F \in \D$, if $G\subseteq F$ then $G\in \D$. An element $F\in \D$ is called a {\it face} of $\D$. The dimension of a face $F$ is $\dim F = |F|-1$ and the dimension of $\D$ is $\dim \D = \max\{\dim F : F\in \D\}$. The faces of $\D$ which are maximal under inclusion are called {\it facets}. We denote the set of the facets of $\D$ by $\FD$. For a simplicial complex $\D$ we can consider a square-free monomial ideal, known as the {\it Stanley-Reisner ideal} of $\D$,
\[\Id = (x_{i_1}\cdots x_{i_s} : \{i_1,\ldots ,i_s\}\notin \D).\]
The $K$-algebra $K[\D]=S/\Id$ is called the {\it Stanley-Reisner ring} of $\D$, and it turns out that 
\[ \dim (K[\D]) = \dim \D + 1.\] 
More precisely, with the convention of denoting by $\wp_A=(x_i:i\in A)$ the prime ideal of $S$ generated by the variables associated to a given subset $A\subseteq [n]$, we have
\[\Id = \bigcap_{F \in \FD} \wp_{[n]\setminus F}. \]
Given any ideal $I\subseteq S$ its $m$th {\it symbolic power} is $I^{(m)}=(I^mS_W)\cap S$, where $W$ is the complement in $S$ of the union of the associated primes of $I$ and $S_W$ denotes the localization of $S$ at the multiplicative system $W$. If $I$ is a square-free monomial ideal then $I^{(m)}$ is just the intersection of the (ordinary) powers of the minimal prime ideals of $I$. Thus
\[\Idm = \bigcap_{F\in \FD} \wp_{[n]\setminus F}^m.\]
The last concept which is needed to understand the main theorem of the paper is a {\it matroid}. A simplicial complex $\D$ on $[n]$ is said to be a matroid if, for any two facets $F$ and $G$ of $\D$ and any $i\in F$, there exists a $j\in G$ such that $(F\setminus \{i\})\cup \{j\}$ is a facet of $\D$. It is well known that if $\D$ is a matroid then $K[\D]$ is Cohen-Macaulay. In particular all the facets of a matroid have the same dimension. An useful property of matroids is the following.

\vspace{1mm}

{\it Exchange property}. Let $\D$ be a matroid on $[n]$. For any two facets $F$ and $G$ of $\D$ and for any $i\in F$, there exists $j \in G$ such that both $(F\setminus \{i\})\cup \{j\}$ and $(G\setminus \{j\})\cup \{i\}$ are facets of $\Delta$.

\vspace{1mm}

\subsection{Statement and proof}

What we are going to prove is the following theorem.
\begin{thm}\label{main}
Let $\D$ be a simplicial complex on $[n]$. Then $S/\Idm$ is Cohen-Macaulay for any $m\in \NN_{\geq 1}$ if and only if $\D$ is a matroid.
\end{thm}
%

\begin{remark}
Notice that Theorem \ref{main} does not depend on the characteristic of $K$.
\end{remark}

\begin{remark}
If $\D$ is the $k$-skeleton of the $(n-1)$-simplex, $-1\leq k\leq n-1$, then $\D$ is a matroid. So Theorem \ref{main} implies that all the symbolic powers of $\Id$ are Cohen-Macaulay.
\end{remark}

In order to prove Theorem \ref{main} it is useful to introduce another square-free monomial ideal associated to a simplicial complex $\D$, namely the cover ideal of $\D$
\[\JD = \bigcap_{F\in \FD} \wp_F.\]
We have $\dim(S/\JD)=n-\dim \D - 1$. The name ``cover ideal" comes from the following fact: A subset $A\subseteq [n]$ is called a vertex cover of $\D$ if $A\cap F \neq \emptyset$ for any $F\in \FD$. Then it is easy to see that
\[\JD = (x_{i_1}\cdots x_{i_s} : \{i_1,\ldots ,i_s\}\mbox{ is a vertex cover of }\D).\]
Let $\Delta^c$ be the simplicial complex on $[n]$ whose facets are $[n]\setminus F$ such that $F\in \FD$. Clearly we have $I_{\D^c}=\JD$ and $I_{\D}=J(\D^c)$. Furthermore $(\D^c)^c=\D$, and it is known that $\D$ is a matroid if and only if $\D^c$ is a matroid (\cite[Theorem 2.1.1]{Ox}). Actually the matroid $\D^c$ is known as the dual of $\D$.
%


In order to have a good combinatorial description of $\JDm$ we need a concept that is more general than vertex cover: For a natural number $k$, a $k$-cover of $\D$ is a nonzero function 
\[\aa:[n]\longrightarrow \NN\]
such that $\sum_{i\in F}\aa(i)\geq k$ for any $F\in \FD$. Of course vertex covers and $1$-covers with values on $\{0,1\}$ are the same things. It is not difficult to see that 
\[\JDm = (x_1^{\aa(1)}\cdots x_n^{\aa(n)}:\aa \mbox{ is an $m$-cover of }\D).\]
A $k$-cover $\aa$ of $\D$ is said to be {\it basic} if for any nonzero function $\bb:[n]\longrightarrow \NN$ with $\bb(i)\leq \aa(i)$ for any $i\in [n]$, if $\bb$ is a $k$-cover of $\D$ then $\bb = \aa$. Of course to the basic $m$-covers of $\D$ corresponds a minimal system of generators of $\JDm$.  

Now let us consider the multiplicative filtration $\J(\D)=\{\JDm\}_{m\in \NN_{\geq 1}}$. We can form the Rees algebra of $S$ with respect to the filtration $\J(\D)$,
\[A(\D) = S \oplus (\bigoplus_{m\geq 1}\JDm).\]
In \cite[Theorem 3.2]{HHT}, Herzog, Hibi and Trung proved that $A(\D)$ is noetherian. In particular, the associated graded ring of $S$ with respect to $\J(\D)$
\[G(\D) = S/\JD \oplus (\bigoplus_{m\geq 1}\JDm/\JD^{(m+1)})\]
and the special fiber
\[\AD = A(\D)/\mm A(\D) = G(\D)/\mm G(\D)\]
are noetherian too. The algebra $A(\D)$ is known as the vertex cover algebra of $\D$, and its properties have been intensively studied in \cite{HHT}. The name comes from the fact that, writing 
\[A(\D)=S \oplus (\bigoplus_{m\geq 1}\JDm \cdot t^m) \subseteq S[t]\]
and denoting by $(A(\D))_m = \JDm \cdot t^m$, it turns out that a (infinite) basis for $A(\D)_m$ as a $K$-vector space is 
\[\{x_1^{\aa(1)}\cdots x_n^{\aa(n)}\cdot t^m : \aa \mbox{ is a $m$-cover of }\D\}.\]
The algebra $\AD$, instead, is called the algebra of basic covers of $\D$, and its properties have been studied by the author with Benedetti and Constantinescu in \cite{BCV} and with Constantinescu in \cite{CV} for a $1$-dimensional simplicial complex $\D$. Clearly, the grading defined above on $A(\D)$ induces a grading on $\AD$, and it turns out that a basis for $(\AD)_m$, $m\geq 1$, as a $K$-vector space is
\[\{x_1^{\aa(1)}\cdots x_n^{\aa(n)}\cdot t^m : \aa \mbox{ is a basic $m$-cover of }\D\}.\]
Notice that if $\aa$ is a basic $m$-cover of $\D$ then $\aa(i)\leq m$ for any $i\in [n]$. This implies that $(\AD)_m$ is a finite $K$-vector space for any $m\in \NN$. So we can speak about the Hilbert function of $\AD$, denoted by $\HF$, and from what said above we have, for $k\geq 1$, 
\[\HF(k)=|\{\mbox{basic $k$-covers of }\D\}|.\] 
The key to prove Theorem \ref{main} is to compute the dimension of $\AD$. So we need a combinatorial description of $\dim(\AD)$. Being in general non-standard graded, the algebra $\AD$ could not have a Hilbert polynomial. However by \cite[Corollary 2.2]{HHT} we know that there exists $h\in \NN$ such that $(\JD^{(h)})^m=\JD^{(hm)}$ for all $m\geq 1$. It follows that $\AD^{(h)} =  \oplus_{m\in \NN} (\AD)_{hm}$ is a standard graded $K$-algebra. Notice that if a set $\{ f_1,\ldots ,f_q \}$ generates $\AD$ as a $K$-algebra then the set $\{f_1^{i_1}\cdots f_q^{i_q} \ : \ 0\leq  i_1,\ldots ,i_q \leq h-1\}$ generates $\AD$ as a $\AD^{(h)}$-module. Thus $\dim (\AD) = \dim (\AD^{(h)})$. Since $\AD^{(h)}$ has a Hilbert polynomial, we get a useful criterion to compute the dimension of $\AD$. First remind that, for two functions $f,g:\NN \rightarrow \mathbb{R}$, the writing $f(k)=O(g(k))$ means that there exists a positive real number $\lambda$ such that $f(k)\leq \lambda \cdot g(k)$ for $k\gg 0$. Similarly, $f(k)=\Omega(g(k))$ if there is a positive real number $\lambda$ such that $f(k)\geq \lambda \cdot g(k)$ for $k\gg 0$

\vspace{1mm}

{\it Criterion for detecting the dimension of $\AD$}. If $\HF(k)=O(k^{d-1})$ then $\dim(\AD)\leq d$. If $\HF(k)=\Omega(k^{d-1})$ then $\dim(\AD)\geq d$. 
\vspace{1mm}

The following proposition justifies the introduction of $\AD$.
\begin{prop}\label{thekey}
For any simplicial complex $\D$ on $[n]$ we have
\[ \dim (\AD) = n - \min \{\depth(S/\JDm) : m\in \NN_{\geq 1}\}\]
\end{prop}
\begin{proof}
Consider $G(\D)$, the associated graded ring of $S$ with respect to $\J(\D)$. Since $G(\D)$ is noetherian, it follows by Bruns and Vetter \cite[Proposition 9.23]{BrVe} that 
\[\min \{\depth(S/\JDm) : m\in \NN_{\geq 1}\} = \grade (\mm G(\D)).\]
We claim that {\it $G(\D)$ is Cohen-Macaulay.} In fact the Rees ring of $S$ with respect to the filtration $\J(\D)$, namely $A(\D)$, is Cohen-Macaulay by \cite[Theorem 4.2]{HHT}.
Let us denote by $A(\D)_+ = \oplus_{m>0}\JDm$ and by $\M = \mm \oplus A(\D)_+$ the unique bi-graded maximal ideal of $A(\D)$. The following short exact sequence
\[0\longrightarrow A(\D)_+ \longrightarrow A(\D) \longrightarrow S \longrightarrow 0\]
yields the long exact sequence on local cohomology
\[\ldots \rightarrow H_{\M}^i(A(\D)_+) \rightarrow H_{\M}^i(A(\D))\rightarrow H_{\M}^i(S)\rightarrow H_{\M}^{i+1}(A(\D)_+)\rightarrow H_{\M}^{i+1}(A(\D))\rightarrow \ldots.\]
By the independence of the base in computing local cohomology modules we have $H_{\M}^i(S)=H_{\mm}^i(S)=0$ for any $i<n$. Furthermore $H_{\M}^i(A(\D))=0$ for any $i\leq n$ since $A(\D)$ is a Cohen-Macaulay $(n+1)$-dimensional (see \cite[Theorem 4.5.6]{BH}) ring. Thus $H_{\M}^i(A(\D)_+)=0$ for any $i\leq n$ by the above long exact sequence. Now let us look at the other short exact sequence
\[0\longrightarrow A(\D)_+(1) \longrightarrow A(\D) \longrightarrow G(\D) \longrightarrow 0,\]
where $A(\D)_+(1)$ means $A(\D)_+$ with the degrees shifted by $1$, and the corresponding long exact sequence on local cohomology
\[\ldots \rightarrow H_{\M}^i(A(\D)_+(1)) \rightarrow H_{\M}^i(A(\D))\rightarrow H_{\M}^i(G(\D))\rightarrow H_{\M}^{i+1}(A(\D)_+(1))\rightarrow  \ldots.\]
Because $A(\D)_+$ and $A(\D)_+(1)$ are isomorphic $A(\D)$-module, $H_{\M}^i(A(\D)_+(1))=0$ for any $i\leq n$. Thus $H_{\M}^i(G(\D))=0$ for any $i<n$. Since $G(\D)$ is a $n$-dimensional ring (see \cite[Theorem 4.5.6]{BH}) this implies, using once again the independence of the base in computing local cohomology, that $G(\D)$ is Cohen-Macaulay. 

Since $G(\D)$ is Cohen-Macaulay $\grade (\mm G(\D)) = \height(\mm G(\D))$. So, because $\AD = G(\D)/\mm G(\D)$, we get
\[ \dim (\AD) = \dim(G(\D))-\height(\mm G(\D))=n-\grade(\mm G(\D)), \]
and the statement follows.
\end{proof}

We are almost ready to show Theorem \ref{main}. We need just a technical lemma which allows us to construct ``many" basic covers.
\begin{lemma}\label{technicalextimation}
Let $s\geq -1$ and $d$ be integer numbers such that $s\leq d-3$. For any positive integer $k$ consider the set
\begin{displaymath}
\begin{array}{lll}
A_k & = & \{(a_1,a_2,\ldots,a_d,b_1,b_2,\ldots,b_{d-s-1})\in \NN^{2d-s-1} \ : \\
&  & a_1+\ldots + a_d = k, \ \ a_1 + \ldots + a_{d-s+1} = b_1 + \ldots + b_{d-s-1},\\
&  & a_1 \geq a_2 \geq \ldots \geq a_d, \mbox{ \ and \ }  b_1, b_2, \ldots , b_{d-s-1} \geq a_2\}.
\end{array}
\end{displaymath}
Then $|A_k|=\Omega(k^{2d-s-3})$.
\end{lemma}
\begin{proof}
Let us set 
\[\displaystyle X_k = \left\{a_1\in \NN \ : \ \frac{(d+1)k}{d+2} \leq a_1 \leq \frac{(d+2)k}{d+3}\right\}.\] 
Of course, setting $\lambda_1=\displaystyle \frac{1}{(d+2)(d+3)}$, we have $|X_k| \geq \lambda_1 \cdot k$.

For a fixed $a_1\in X_k$, set
\[Y_k(a_1) = \{(a_2,\ldots ,a_d) \ : \ a_1+a_2+\ldots +a_d=k\}\]
The vectors $(a_2,\ldots ,a_d)\in Y_k(a_1)$ are so many as the integer partitions of $k-a_1$ with at most $d-1$ parts. Because $a_1\in X_k$ these are at least so many as the partitions $\lfloor k/(d+3) \rfloor$ with at most $d-1$ parts. These, in general, are less than all the monomials of degree $\lfloor k/(d+3) \rfloor$ in $d-1$ variables, i.e. $\displaystyle \binom{d-2+\lfloor k/(d+3) \rfloor}{d-2}$, since a permutation of the variables gives the same partitions but may give different monomials. Anyway, since this is the only reason, the number of the possible $(a_2,\ldots ,a_d)$ is at least  
\[ \displaystyle \frac{1}{(d-1)!} \binom{d-2+\lfloor k/(d+3) \rfloor}{d-2}.\]
So there exists a positive real number $\lambda_2$, independent on $a_1$, such that $|Y_k(a_1)|\geq \lambda_2 \cdot k^{d-2}$.

Let ${\bf a}=(a_1,a_2,\ldots ,a_d)$ be a vector such that $a_1\in X_k$ and $(a_2,\ldots ,a_d)\in Y_k(a_1)$. Then set
\[Z_k({\bf a}) = \{(b_1,\ldots ,b_{d-s-1})\in \NN_{\geq a_2}^{d-s-1} \ : \ b_1+\ldots + b_{d-s-1}=a_1+\ldots +a_{d-s-1}\}\]
It is easy to notice that the vectors $(b_1,\ldots ,b_{d-s-1})\in Z_k({\bf a})$ are so many as all the monomials of degree $a_1+\ldots +a_{d-s-1} - (d-s-1)a_2$ in $d-s-1$ variables. Clearly we have 
\[a_1+\ldots +a_{d-s-1} - (d-s-1)a_2\geq a_1 - (d-s-1)a_2.\]
But $\displaystyle a_2 \leq k - a_1\leq \frac{k}{d+2}$. So we get
\[\displaystyle a_1+\ldots +a_{d-s-1} - (d-s-1)a_2\geq a_1 - (d-s-1)a_2\geq \frac{(d+1)k}{d+2}-\frac{dk}{d+2}=\frac{k}{d+2}.\]
So the elements of $Z_k({\bf a})$ are at least so many as the monomials of degree $\lfloor k/(d+2) \rfloor$ in $d-s-1$ variables. Therefore there is a positive real number $\lambda_3$, not depending on ${\bf a}$, such that $|Z_k({\bf a})|\geq \lambda_3 \cdot k^{d-s-2}$.

Finally, we have that
\[|A_k|\geq \sum_{a_1 \in X_k}\sum_{(a_2,\ldots ,a_d)\in Y_k(a_1)}|Z_k({\bf a})| \geq (\lambda_1\cdot k)\cdot (\lambda_2\cdot k^{d-2})\cdot (\lambda_3 \cdot k^{d-s-2}) = \lambda_1 \lambda_2 \lambda_3 \cdot k^{2d-s-3}.\]
\end{proof}

Now we are ready to prove Theorem \ref{main}.

\begin{proof}
By the duality on the matroids it is enough to prove that {\it $S/\JDm$ is Cohen-Macaulay for any $m\in \NN_{\geq 1}$ if and only if $\D$ is a matroid.} Suppose that $\D$ is $(d-1)$-dimensional.

\vspace{1mm}

{\it If-part.} Let us consider a basic $k$-cover $\aa$ of $\D$. Let $F$ be a facet of $\D$ such that $\sum_{j\in F}\aa(j)=k$ ($F$ exists because $\aa$ is basic). Set 
\[A_F=\{\aa(j) \ : \ j\in F\}.\]
We claim that {\it for any $i\in [n]$ we have $\aa(i)\in A_F$}. In fact, if $i_0\in [n]$ does not belong to $F$, then, because $\aa$ is basic, there exists a facet $G$ of $\D$ such that $i_0\in G$ and $\sum_{i\in G}\aa(i)=k$. By the exchange property there exists a vertex $j_0\in F$ such that $(G\setminus \{i_0\})\cup \{j_0\}$ and $(F\setminus \{j_0\})\cup \{i_0\}$ are facets of $\Delta$. But 
\[\sum_{i\in (G\setminus \{i_0\})\cup \{j_0\}}\aa(i) \geq k \ \implies \aa(j_0)\geq \aa(i_0),\]
and
\[\sum_{j\in (F\setminus \{j_0\})\cup \{i_0\}}\aa(j) \geq k \ \implies \aa(i_0)\geq \aa(j_0).\]
Hence $\aa(i_0)=\aa(j_0)\in A_F$. The number of ways to give values on vertices of $F$ such that the sum of the values on the whole $F$ is $k$ are $\displaystyle \binom{k+d-1}{d-1}$. This implies that, for $k\geq 1$,
\[\displaystyle \HF(k) = |\{\mbox{basic $k$-covers of }\D\}|\leq |\FD| \cdot \binom{k+d-1}{d-1} \leq \binom{n}{d}\cdot \binom{k+d-1}{d-1}.\]
So $\HF(k)=O(k^{d-1})$, therefore $\dim(\AD)\leq d$. But $\dim(S/J(\D))=n-d$, so by Proposition \ref{thekey} 
\[d \geq \dim(\AD) = n - \min \{\depth(S/\JDm):m\in \NN_{\geq 1}\} \geq d,\]
from which $S/\JDm$ is Cohen-Macaulay for any $m\in \NN_{\geq 1}$.

\vspace{1mm}

{\it Only if-part.} Suppose contrary that $\D$ is not a matroid. Then there exist two facets $F$ and $G$ of $\D$ and a vertex $i\in F$ such that $(F\setminus \{i\})\cup \{j\}$ is not a facet of $\D$ for any $j\in G$. Let $s$ be the greatest integer such that there exists an $s$-dimensional subface $F'$ of $F\setminus \{i\}$ such that there is a $(d-s-2)$-dimensional subface of $G$ whose union with $F'$ is a facet of $\D$. Notice that $s\leq d-3$ and $s$ might be $-1$. Let $F_0\subseteq F\setminus \{i\}$ be an $s$-dimensional face and $G_0\subseteq G$ a $(d-s-2)$-dimensional face satisfying the above conditions. Let $(a_1,\ldots ,a_d,b_1,\ldots ,b_{d-s-1})\in A_k$, where $A_k$ is the set defined in Lemma \ref{technicalextimation}.
Set $F=\{i_1,\ldots,i_d\}$ with $i_1=i$ and $F_0=\{i_{d-s},\ldots ,i_d\}$. Also, set $G=\{j_1,\ldots ,j_d\}$ where $G_0=\{j_1,\ldots ,j_{d-s-1}\}$. Now we define the following numerical function on $[n]$:
\begin{displaymath}
\aa'(v) = \left\{ \begin{array}{ll} a_p & \mbox{if } v=i_p\\
b_q &  \mbox{if } v=j_q \mbox{ and } q< d-s\\
k & \mbox{otherwise} \ \end{array}  \right.
\end{displaymath}
We claim that $\aa'$ is a $k$-cover, not necessarily basic. By the definition of $\aa'$ we have to check that for any facet $H$ of $\D$ contained in $F\cup G_0$ we have the inequality $\sum_{h\in H}\aa'(h)\geq k$. If $i\notin H$, then $G_0\subset H$ by the maximality of $s$. But then we have
\[\sum_{h\in H}\aa'(h)=\sum_{h\in G_0}\aa'(h)+\sum_{h\in H\setminus G_0}\aa'(h)\geq \sum_{h\in G_0}\aa'(h)+\sum_{h\in F_0}\aa'(h)=k.\]
If $i\in H$, then we have
\begin{displaymath}
\begin{array}{ccl}
\sum_{h\in H}\aa'(h) & = & a_1+\sum_{h\in H\cap (F\setminus \{i\})}\aa'(h)+\sum_{h\in H\setminus F}\aa'(h)\\
& \geq & a_1+\sum_{h\in H\cap (F\setminus \{i\})}\aa'(h)+|H\setminus F|\cdot a_2\\
& \geq & a_1+\ldots +a_d \ = \ k.  
\end{array}
\end{displaymath}
Reducing the values of $\aa'$ where possible we can make it in a basic $k$-cover $\aa$. However we cannot reduce the values at the vertices of $F\cup G_0$ because the equalities
\[\sum_{h\in F}\aa'(h) = k \mbox{ \ \ and \ \ }\sum_{h\in F_0\cup G_0}\aa'(h) = k.\]
Thus the basic $k$-covers of $\FD$ are at least so many as the cardinality of $A_k$. So by Lemma  \ref{technicalextimation}  there exists a positive real number $\lambda$ such that for $k\gg 0$ we have
\[\displaystyle \HF(k) = |\{\mbox{basic $k$-covers of }\D\}|\geq \lambda \cdot k^{2d-s-3} \geq \lambda \cdot k^d.\]
So $\HF(k)=\Omega(k^d)$, therefore $\dim(\AD)\geq d+1$. Using the Proposition \ref{thekey} we have that
\[\min \{\depth(S/\JDm):m\in \NN_{\geq 1}\}\leq n-d-1, \]
which contradicts the hypothesis that $S/\JDm$ is Cohen-Macaulay for any $m\in \NN_{\geq 1}$.
\end{proof}

We end the paper stating two corollaries of Theorem \ref{main}. First we recall that the multiplicity of a standard graded $K$-algebra $R$, denoted by $e(R)$, is the leading  coefficient of the Hilbert polynomial times $(\dim(R)-1)!$. Geometrically, let $\Proj R \subseteq \PP^N$, i.e. $R=K[X_0,\ldots ,X_N]/J$ for a homogeneous ideal $J$. The multiplicity $e(R)$ counts the number of distinct points of $\Proj R \cap H$, where $H$ is a generic linear subspace of $\PP^N$ of dimension $N - \dim (\Proj R)$.

\begin{corollary}\label{multdim}
A simplicial complex $\D$ is a $(d-1)$-dimensional matroid if and only if
\[\dim(\AD) = \dim (K[\D]) = d.\]
Moreove, if $\D$ is a matroid then 
\[\displaystyle \HF(k)\leq \frac{e(K[\D])}{(\dim(\AD)-1)!}k^{\dim(\AD)-1}+O(k^{\dim(\AD)-2}).\]
\end{corollary}
\begin{proof}
The first fact follows putting together Theorem \ref{main} and Proposition \ref{thekey}.
For the second fact, we have to recall that, during the proof of Theorem \ref{main}, we showed that for a $(d-1)$-dimensional matroid $\D$ we have the inequality
\[\displaystyle \HF(k)\leq |\FD| \cdot \binom{k+d-1}{d-1}.\]
It is well known that if $\D$ is a pure simplicial complex then $|\FD|=e(K[\D])$ (for instance see \cite[Corollary 5.1.9]{BH}), so we get the conclusion.
\end{proof}

\begin{example}
If $\D$ is not a matroid the inequality of Corollary \ref{multdim} may not be true. For instance, take $\D=C_{10}$ the decagon (thus it is a $1$-dimensional simplicial complex). Since $C_{10}$ is a bipartite graph $\bar{A}(C_{10})$ is a standard graded $K$-algebra by \cite[Theorem 5.1]{HHT}. In particular it admits a Hilbert polynomial, and for $k\gg 0$ we have
\[\displaystyle \operatorname{HF}_{\bar{A}(C_{10})}(k)=\frac{e(\bar{A}(C_{10}))}{(\dim(\bar{A}(C_{10}))-1)!}k^{\dim(\bar{A}(C_{10}))-1}+O(k^{\dim(\bar{A}(C_{10}))-2}).\]
In \cite{CV} it is proved that for any bipartite graph $G$ the algebra $\bar{A}(G)$ is a homogeneous algebra with straightening law on a poset described in terms of the minimal vertex covers of $G$. So the multiplicity of $\bar{A}(G)$ can be easily read off from the above poset. In our case it is easy to check that $e(\bar{A}(C_{10}))=20$, whereas $e(K[C_{10}])=10$.
\end{example}

Let us introduce the last result of the paper. An ideal $I$ of a ring $R$ is a set-theoretic complete intersection if there exist $f_1,\ldots ,f_h \in R$, where $h=\height(I)$, such that $\sqrt{(f_1,\ldots ,f_h)}=\sqrt{I}$. The importance of this notion comes from algebraic geometry, since if $I$ is a set-theoretic complete intersection then the variety $\V(I)\subseteq \Spec(R)$ can be defined set-theoretically ``cutting" the ``right" number of hypersurfaces of $\Spec(R)$. A necessary, in general not sufficient, condition for $I$ to be a set-theoretic complete intersection is that the cohomological dimension of it, $\cd(R,I)=\max \{i: H_I^i(R)\neq 0 \}$, is $h$. By a result of Lyubeznik in \cite{Ly} it turns out that $\cd(S,\Id) = n - \depth(K[\D])$, so if $\Id$ is a set-theoretic complete intersection $K[\D]$ will be Cohen-Macaulay.

\begin{remark}
In general if $K[\D]$ is Cohen-Maculay then $\Id$ might not be a set-theoretic complete intersection. For instance, if $\D$ is the triangulation of the real projective plane with 6 vertices described in \cite[p. 236]{BH}, then $K[\D]$ is Cohen-Macaulay whenever $\chara(K)\neq 2$. However, for any characteristic of $K$, $\Id$ need at least (actually exactly) $4$ polynomials of $K[x_1,\ldots ,x_6]$ to be defined up to radical (see the paper of Yan \cite[p. 317, Example 2]{Ya}), but $\height(\Id)=3$.
\end{remark}

\begin{corollary}\label{sci}
Let $K$ be an infinite field. For any matroid $\D$, the ideal $\Id S_{\mm}$ is a set-theoretic complete intersection in $S_{\mm}$.
\end{corollary}
\begin{proof}
By the duality on matroids it is enough to prove that $\JD S_{\mm}$ is a set-theoretic complete intersection.
For $h \gg 0$ it follows by \cite[Corollary 2.2]{HHT} that the $h$th Veronese of $\AD$,  
\[ \AD^{(h)}=\bigoplus_{m\geq 0}\AD_{hm}, \]
is standard graded. Therefore $\AD^{(h)}$ is the ordinary fiber cone of $\JD^{(h)}$. Moreover $\AD$ is finite as a $\AD^{(h)}$-module. So the dimensions of $\AD$ and of $\AD^{(h)}$ are the same. Therefore, using Theorem \ref{main} and Proposition \ref{thekey}, we get
\[ \height(\JD S_{\mm})=\height(\JD)=\dim \AD^{(h)}=\ell(\JD^{(h)})=\ell((\JD S_{\mm})^{(h)}), \]
where $\ell(\cdot)$ is the analytic spread of an ideal, i.e. the Krull dimension of its ordinary fiber cone.
From a result by Northcott and Rees in \cite[p. 151]{NR}, since $K$ is infinite, it follows that the analytic spread of $(\JD S_{\mm})^{(h)}$ is the cardinality of a set of minimal generators of a minimal reduction of $(\JD S_{\mm})^{(h)}$. 
Clearly the radical of such a reduction is the same as the radical of $(\JD S_{\mm})^{(h)}$, i.e. $\JD S_{\mm}$, so we get the statement.
\end{proof}

\begin{remark}
Notice that a reduction of $I S_{\mm}$, where $I$ is a homogeneous ideal of $S$, might not provide a reduction of $I$. So localizing at the maximal irrelevant ideal is a crucial assumption of Corollary \ref{sci}. It would be interesting to know whether $\Id$ is a set-theoretic complete intersection in $S$ whenever $\D$ is a matroid.
\end{remark}

\subsection*{Acknowledgements}
The author would like to thank Margherita Barile for precious information about Corollary \ref{sci}, Aldo Conca for carefully reading the paper and Le Dinh Nam for useful conversations about the topic of the work.


\begin{thebibliography}{99}
\addcontentsline{toc}{chapter}{Bibliografia}
\bibitem[BCV]{BCV} B. Benedetti, A. Constantinescu, M. Varbaro, \textit{Dimension, depth and zero-divisors of the algebra of basic $k$-covers of a graph}, Le Matematiche LXIII, n. II, pp. 117-156, 2008.
\bibitem[BH]{BH} W. Bruns, J. Herzog, \textit{Cohen-Macaulay rings}, Cambridge studies in advanced mathematics, 1993.
\bibitem[BrVe]{BrVe} W. Bruns, U. Vetter, \textit{Determinantal rings}, Lecture notes in mathematics 1327, 1980.
\bibitem[CV]{CV} A. Constantinescu, M. Varbaro, \textit{Koszulness, Krull Dimension and Other Properties of Graph-Related Algebra}, available on line at arXiv:1004.4980v1, 2010.
\bibitem[HHT]{HHT} J. Herzog, T. Hibi, N. V. Trung, \textit{Symbolic powers of monomial ideals and vertex cover algebras}, Adv. in Math. 210, pp. 304-322, 2007.
\bibitem[Ly]{Ly} G. Lyubeznik, \textit{On the local cohomology modules $H_{\mathfrak{U}}^i(R)$ for ideals $\mathfrak{U}$ generated by an $R$-sequence}, ``Complete Intersection'', Lect. Notes in Math. 1092, pp. 214-220, 1984.
\bibitem[MS]{MS} E. Miller, B. Sturmfels, \textit{Combinatorial commutative algebra}, Graduate Texts in Mathematics 227, Springer-Verlag, 2005.
\bibitem[MT]{MT} N. C. Minh, N. V. Trung, \textit{Cohen-Macaulayness of powers of monomial ideals and symbolic powers of Stanley-Reisner ideals}, available on line at arXiv:1003.2152v1, 2010.
\bibitem[NR]{NR} D.G. Northcott, D. Rees, \textit{Reduction of ideals in local rings}, Proc. Cambridge Philos. Soc. 50, pp. 145-158, 1954.
\bibitem[Ox]{Ox} J. G. Oxley, \textit{Matroid Theory}, Oxford University Press, 1992.
\bibitem[St]{St} R. P. Stanley, \textit{Combinatorics and commutative algebra}, Progress in mathematics 41, Birkh\"auser Boston, 1996.
\bibitem[TY]{TY} N. Terai, K. Yoshida, \textit{Locally complete intersection Stanley-Reisner ideals}, available on line at arXiv:0901.3899v1, 2009.
\bibitem[We]{We} D. J. A. Welsh, \textit{Matroid Theory}, Academic Press, London, 1976.
\bibitem[Ya]{Ya} Z. Yan, \textit{An \'etale analog of the Goresky-Macpherson formula for subspace arrangemets}, J. Pure and App. Alg. 146, pp. 305-318, 2000.
\end{thebibliography}
\end{document}